\documentclass[12pt, reqno]{amsart}
\usepackage{amsmath, amsthm, amscd, amsfonts, amssymb, graphicx, color}
\usepackage[bookmarksnumbered, colorlinks, plainpages]{hyperref}
\usepackage{fancyhdr}
\usepackage{cite}
\hypersetup{colorlinks=true,linkcolor=red, anchorcolor=green, citecolor=cyan, urlcolor=red, filecolor=magenta, pdftoolbar=true}

\theoremstyle{definition}

\theoremstyle{remark}

\numberwithin{equation}{section}

\newtheorem{thm}{Theorem}[section]
\newtheorem{prop}[thm]{Proposition}

\newtheorem{lem}[thm]{Lemma}

\numberwithin{equation}{section}
\newcommand{\real}{{\mathbb R}}
\newcommand{\norm}[1]{\left\Vert#1\right\Vert}
\newcommand{\8}{\infty}
\renewcommand{\a}{\mathfrak{a}}

\usepackage{comment}

\begin{document}

\title[dimension-free estimates for...]{dimension-free estimates for the vector-valued variational operators}

\author{Danqing He}
\address{Department of Mathematics, Sun Yat-sen (Zhongshan) University, Guangzhou, 510275, China}
\email{hedanqing@mail.sysu.edu.cn}
\author{Guixiang Hong}
\address{School of Mathematics and Statistics, Wuhan University, Wuhan 430072 and Hubei Key Laboratory of Computational Science, Wuhan University, Wuhan 430072, China}
\email{guixiang.hong@whu.edu.cn}
\author{Wei Liu}
\address{School of Mathematics and Statistics, Wuhan University, Wuhan 430072, China}
\email{Wl.math@whu.edu.cn}
\date{June 10, 2018.}
\begin{abstract}
In this paper, we study dimension-free $L^p$ estimates for UMD lattice-valued $q$-variations of Hardy-Littlewood averaging
operators associated with the Euclidean balls.
\end{abstract}
\maketitle

\bigskip


\section{Introduction}\label{ST1}



Stein  \cite{E.M.S1}
obtained the first dimension-free result  for the (ball) Hardy-Littlewood maximal operator,
as an application of the boundedness of spherical maximal operators; see \cite{SS} for more details of Stein's argument.
This inspired a lot of generalizations  for Hardy-Littlewood maximal operators related to other convex bodies, in particular
$\ell^r$ balls; see, for instance, \cite{Bour1}, \cite{Bour2}, \cite{Bour3},   \cite{Car}, and \cite{D.M}.
Two interesting results in this direction obtained recently were due to Bourgain,   Mirek,  Stein, and Wr\'{o}bel \cite{J.M.E.B, J.M.E.B1}, where
the authors proved the dimension-free estimates for $q$-variations of Hardy-Littlewood averaging operator defined over symmetric convex bodies in $\real^d$, and in $\mathbb{Z}^d$ respectively.

Variational inequalities call a lot of attention in past decades for various reasons, one of which is that we can obtain
pointwise convergence of the related sequence of operators without recalling a dense subspace, an advantage compared with the corresponding strategy using maximal operators.
Variational inequalities in most cases are much more difficult to be handled than the maximal inequalities since variational semi-norms (defined later) involve not only the size information but also the oscillation information of the underlying sequence of operators, 
which makes that it  always a challenging task to strengthen a maximal estimate to a variational one. The first variational result, to our knowledge, was
obtained by L\'{e}pingle \cite{D.L} for martingales, which was later reproved by  Pisier and Xu~\cite{G.P}.
Many interesting results in harmonic analysis follow from this direction (cf.e.g., \cite{JRKM1, YCC, AJ, TMC, RRJ, RJM, CQ, RG, AM, AX, AX1, RATCJ}), among which we highlight  \cite{Bour} due to Bourgain and \cite{RAJ} by Jones et al.


Vector-valued inequalities are natural generalizations of the scale-valued ones, one of which
in harmonic analysis is the famous Fefferman-Stein inequality \cite{FS}.
So is it possible to generalize \cite{J.M.E.B} to the vector-valued setting?
The study of vector-valued variational inequalities was initiated by
Pisier and Xu~\cite{G.P}.
More vector-valued variational inequalities were obtained in 
\cite{TMI, MA, MA1}.
  Ma and the second author  established vector-valued variational inequalities associated to differentiation ergodic averages and symmetric diffusion semigroups, where the functions take values in any Banach space of martingale cotype $q_0$ \cite{GXHTM}, and in the Banach lattice \cite{GXHTM1}.%

In this note we answer the question in previous paragraph by obtaining  dimension-free estimates for the
vector-valued variational averaging operators.
Our result generalizes also Deleaval and Kriegler \cite{D.K}, where
the dimension-free version of the Fefferman-Stein inequality was   obtained.

To state our results, let us recall the definition of the $q$-variation. For $q\in[1,\8)$ the $q$-variation seminorm $V_q$ of a complex-valued function $(0,\8)\times\real^d \ni(t,x)\mapsto \a_t(x)$ is defined by
$$V_q(\a_t(x):t\in Z)=\sup_{\substack{0<t_0<\cdots<t_J\\ t_j\in Z}}\bigg(\sum_{j=0}^J|\a_{t_{j+1}}(x)-\a_{t_j}(x)|^q\bigg)^{1/q},$$
where $Z$ is a subset of $(0,\8)$ and the supremum is taken over all finite increasing sequences in $Z$.
In order to avoid some problems with measurability of $V_q(\a_t(x):t\in Z)$ we assume that $(0,\8)\ni t\mapsto\a_t(x)$ is always a continuous function for every $x\in \real^d$.

We recall the necessary preliminaries on UMD lattices, i.e. Banach lattices which are UMD spaces (see \cite{G.P1}). The UMD lattice $X$ can be represented as a lattice consisting of (equivalence classes of) measurable functions on some measure space $(\Omega,\mu)$. The $X$-valued functions on $\real^d$ can be viewed as scalar-valued functions on $\real^d\times \Omega$. 
We refer the readers to the nice book by Lindenstrauss and Tzafriri \cite{J.L} for more information on Banach lattices. Throughout the paper, $X$ denotes a UMD lattice.%

Let $f$ be a $X$-valued locally integrable function which is defined on $\real^d$. For $t>0$, let $B_t$ denote the open ball centered at the origin 0 with radius $r(B_t)$ equal to $t$. Then we define
$$A_tf(x,\omega)=\frac{1}{|B_t|}\int_{B_t}f(x-y,\omega)dy,\,\,x\in\real^d,\,\omega\in\Omega$$
where $|\cdot|$ denotes the Lebesgue measure. These are the central averaging operators on $\real^d$. The $X$-valued $q$-variation of the family of averaging operators $(A_t)_{t>0}$ is defined as
$$V_q(A_tf(x,\omega):t>0)=\sup_{\substack{0<t_0<\cdots<t_J\\ t_j\in Z}}\bigg(\sum_{j=0}^J|A_{t_{j+1}}f(x,\omega)-A_{t_j}f(x,\omega)|^q\bigg)^{1/q},$$
where $Z$ is a subset of $(0,\8)$ and the supremum is taken over all finite increasing sequences in $Z$.

\medskip

The main result of this paper is the following theorem.%

\begin{thm}\label{main thm}
Let $X$ be a UMD lattice. Let $1<p<\8$ and $2<q<\8$. Then there exists a constant $C_{p,q,X}>0$ independent of the dimension $d$ such that
\begin{equation*}\label{main inequalities}
  \norm{V_q(A_tf(x,\omega):t>0)}_{L^p(\real^d;X)}\leq C_{p,q,X}\norm{f}_{L^p(\real^d;X)},\quad \forall f\in L^p(\real^d;X).
\end{equation*}
\end{thm}

If the constants $C_{p,q,X}$ in inequality (\ref{main inequalities}) are allowed to depend on $d$, then this result has essentially already been known in \cite{GXHTM1} due to the second author and Ma, which follows from the weighted norm inequalities \cite{MA, MA1} through the application of Rubio de Francia's extrapolation theorem. The point of this result lies in the fact that the constant $C_{p,q,X}>0$ can be taken to be independent of the dimension $d$. Then, the difficulty becomes apparent since as far as we know almost all the weighted norm inequalities are dimension dependent except the ones for radial weights in \cite{CrSo13} and there does not exist a version of dimension-free extrapolation theorem. Motivated by the scalar-valued result \cite{J.M.E.B}, the key idea is to exploit the UMD-lattice valued variational estimates for semigroup established in \cite{GXHTM1} and the dimension-free UMD-lattice valued square function estimates. The idea in the proof of the latter can be traced back to the one due to Stein, see e.g. \cite{SS}.

To prove Theorem \ref{main thm}, we need the following standard method of dealing with $q$-variation.%

We will handle $V_q(A_tf(x,\omega):t>0)$ by dividing it into long and short variations. Fixed an increasing sequence $(t_i)_{i>0}$.
~For each interval $I_i=(t_i,t_{i+1}]$, first we consider two cases:
\medskip

\begin{itemize}
  \item Case 1: $I_i$ does not contain any integral power of 2;
  \item Case 2: $I_i$ contains integral powers of 2.
  \end{itemize}

\medskip
In Case 1, for interval $I_i$, there are some $k\in \mathbb{Z}$ such that $I_i\subset(2^k,2^{k+1}]$. In Case 2, letting $m_i=\min\{k :2^k\in I_i\}$ and $n_i=\max\{k:2^k\in I_i\}$, we divide $I_i$ into three subintervals: $(t_i,2^{m_i}]$, $(2^{m_i},2^{n_i}]$ and $(2^{n_i},t_{i+1}]$ (noting that if $m_i=n_i$, the middle interval is empty). Then we introduce two collections of intervals:
\medskip

\begin{itemize}
  \item $\mathcal{S}$ consists of all intervals in Case 1, and $(t_i,2^{m_i}]$, $(2^{n_i},t_{i+1}]$ in Case 2 and $\mathcal{S}_k$ consists of all intervals in $\mathcal{S}$ and contained in $(2^k,2^{k+1}]$;
  \item $\mathcal{L}$ consists of all intervals $(2^{m_j},2^{n_j}]$ in Case 2.
\end{itemize}
\medskip

Note that $\mathcal{S}$, $\mathcal{L}$ are two disjoint families of intervals. Hence we have
\begin{equation*}
 \begin{split}
 &\big(\sum_i|A_{t_{i+1}}f(x,\omega)-A_{t_i}f(x,\omega)|^q\big)^{\frac{1}{q}}\leq C_q\big(\sum_{\substack{i,\\{I_i=(t_i,t_{i+1}]\in\mathcal{S}}}}|A_{t_{i+1}}f(x,\omega)-A_{t_i}f(x,\omega)|^q\big)^{\frac{1}{q}}\\
&+C_q\big(\sum_{\substack{i,\\I_i=(t_i,t_{i+1}]\in\mathcal{L}}}|A_{t_{i+1}}f(x,\omega)-A_{t_i}f(x,\omega)|^q\big)^{\frac{1}{q}}=C_q(I+II).
\end{split}
\end{equation*}

\medskip

The first term on the right hand side is controlled by
\begin{align*}
 I&\leq\big(\sum_{k\in \mathbb{Z}}\sup_{(t_i)_i\in\mathcal{S}_k}\sum_{i}|A_{t_{i+1}}f(x,\omega)-A_{t_i}f(x,\omega)|^2\big)^{\frac{1}{2}}\\
 &\leq \big(\sum_{k\in \mathbb{Z}}V_2(A_tf(x,\omega):t\in(2^k,2^{k+1}])^2\big)^{1/2}
\end{align*}
which is denoted by $\mathcal{SV}_2(\mathcal{A})f$. While the second term is controlled by
$$II\leq V_q(A_tf(x,\omega):t\in\{2^k:k\in \mathbb{Z}\})$$
which is denoted by $\mathcal{LV}_q(\mathcal{A})f$.%

\medskip

The paper is organized as follows. In Section \ref{ST2}, we estimate the short variation $\mathcal{SV}_2(\mathcal{A})f$. 
The long variation $\mathcal{LV}_q(\mathcal{A})f$ is treated in Section \ref{ST3} by appealing to the known UMD-lattice valued $q$-variational estimate for semigroups. 



\bigskip
\section{the estimate of the short variation}\label{ST2}
In this section we estimate the short variation $\mathcal{SV}_2(\mathcal{A})f$.%

\begin{thm}\label{S-estimate}
Let $X$ be a UMD lattice. Let $1<p<\8$, then there exists a constant $C_{p,X}>0$ independent of the dimension $d$ such that
\begin{equation}\label{S-inequalities}
  \norm{\mathcal{SV}_2(\mathcal{A})f(x,\omega)}_{L^p(\real^d;X)}\leq C_{p,X}\norm{f(x,\omega)}_{L^p(\real^d;X)},\quad\forall f\in L^p(\real^d;X).
\end{equation}
\end{thm}

Let $\mathcal{S}(\real^d)$ denote the set of Schwartz functions on $\real^d$. Let $f\in \mathcal{S}(\real^d)\otimes X$,  where $f=\sum_{i=1}^nx_if_i$, $f_i\in  \mathcal{S}(\real^d)$, $x_i\in X$. By a simple density argument, it suffices to establish Theorem \ref{S-estimate} for all $f\in\mathcal{S}(\real^d)\otimes X$.

The estimate of short variation $\mathcal{SV}_2(\mathcal{A})f$ will be based on the following lemma.
\begin{lem}\label{S-lemma}
Given a $k\in \mathbb{Z}$ and a differentiable function $\phi:(2^k,2^{k+1}]\rightarrow[0,\8)$, then we have
\begin{equation}\label{short variation estimate}
V_2(\phi_t:t\in(2^k,2^{k+1}])\leq \big(\int^{2^{k+1}}_{2^k}|t\frac{d}{dt}\phi_t|^2\frac{dt}{t}\big)^{\frac{1}{2}}.
\end{equation}
\begin{proof}
Let $2^k<t_0<t_1<\cdots<t_J\leq 2^{k+1}$, then by the Cauchy-Schwarz inequality we have
$$|\phi_{t_{i+1}}-\phi_{t_i}|^2
=|\int_{t_i}^{t_{i+1}}\frac{d\phi_t}{dt}dt|^2\leq(t_{i+1}-t_i)\int_{t_i}^{t_{i+1}}|\frac{d\phi_t}{dt}|^2dt
\leq\int_{t_i}^{t_{i+1}}|t\frac{d\phi_t}{dt}|^2\frac{dt}{t},$$
where we use the equality $t_{i+1}-t_i\le 2^k$. Hence
$$V_2(\phi_t:t\in(2^k,2^{k+1}])^2\leq \int_{2^k}^{2^{k+1}}|t\frac{d\phi_t}{dt}|^2\frac{dt}{t}.$$
The lemma is proved.
\end{proof}
\end{lem}%

\medskip

By Lemma \ref{S-lemma}, we have
$$\big(\sum_{k\in \mathbb{Z}}V_2(A_tf(x,\omega):t\in(2^k,2^{k+1}])^2\big)^{1/2}\leq {\big(\int^{\8}_0|t\frac{d}{dt}A_tf|^2\frac{dt}{t}\big)^{{1}/{2}}},$$
then Theorem \ref{S-estimate} reduces to prove
\begin{equation}\label{estimate for square function}
\|{\big(\int^{\8}_0|t\frac{d}{dt}A_tf|^2\frac{dt}{t}\big)^{\frac{1}{2}}}\|_{L^p(\real^d;X)}\leq C_{p,X}\norm{f}_{L^p(\real^d;X)}.
\end{equation}
Denote
$$S_tf(x,\omega)=\int_{\mathbb{S}^{d-1}}f(x-ty,\omega)d\sigma(y),$$
where $d\sigma$ denotes the normalized Haar measure on $\mathbb{S}^{d-1}$. Then by the coordinate formula, we have
\begin{align*}
A_tf(x,\omega)&=\frac{1}{|B_t|}\int_{B_t}f(x-y,\omega)dy\\
              &=\frac{\int_{B_t}f(x-y,\omega)dy}{\int_{|y|\leq t}dy}\\
              &= \frac{d}{t^d}\int_0^t\int_{\mathbb{S}^{d-1}}f(x-ry,\omega)r^{d-1}d\sigma(y)dr\\
              &=d\int_0^1\int_{\mathbb{S}^{d-1}}f(x-rty,\omega)r^{d-1}d\sigma(y)dr\\
              &=d\int_0^1r^{d-1}S_{tr}f(x,\omega)dr.
\end{align*}
Then, we obtain
$$\frac{d}{dt}A_tf(x,\omega)=d\int_0^1r^{d-1}\frac{d}{dt}S_{tr}f(x,\omega)dr.$$
Using Minkowski's integral inequality, we have
\begin{align*}
  \bigg(\int^{\8}_0|t\frac{d}{dt}A_tf(x,\omega)|^2\frac{dt}{t}\bigg)^{\frac{1}{2}}&=d\bigg(\int_0^{\8}|t\int_0^1r^{d-1}
\frac{d}{dt}S_{tr}f(x,\omega)dr|^2\frac{dt}{t}\bigg)^\frac{1}{2}\\
&\leq d\int_0^1r^{d-1}\bigg(\int_0^{\8}|t\frac{d}{dt}S_{tr}f(x,\omega)|^2\frac{dt}{t}\bigg)^\frac{1}{2}dr\\
&\leq \bigg(\int_0^{\8}|t\frac{d}{dt}S_{t}f(x,\omega)|^2\frac{dt}{t}\bigg)^\frac{1}{2}.
\end{align*}
To prove inequality (\ref{estimate for square function}), it is enough to prove
\begin{equation}\label{eatimate for square function}
\|{\big(\int^{\8}_0|t\frac{d}{dt}S_tf|^2\frac{dt}{t}\big)^{\frac{1}{2}}}\|_{L^p(\real^d;X)}\leq C_{p,X}\norm{f}_{L^p(\real^d;X)}.
\end{equation}

According to the above discussions, to prove Theorem \ref{S-estimate}, we only need to prove the following proposition.
\begin{prop}\label{S-lem}
Let $1<p<\8$ and $X$ be a UMD lattice. Then there exists $d_0(p, X)\in\mathbb{N}$ such that for $d\geq d_0(p,X)$ and every $f\in L^p(\real^d;X)$, we have
$$\|{\big(\int^{\8}_0|t\frac{d}{dt}S_tf|^2\frac{dt}{t}\big)^{\frac{1}{2}}}\|_{L^p(\real^d;X)}\leq C_{p,X}\norm{f}_{L^p(\real^d;X)},$$
where the constant $C_{p,X}>0$ independent of the dimension $d$.
\end{prop}

Indeed, Proposition \ref{S-lem} is enough to prove Theorem \ref{S-estimate}. Since, for $0<d<d_0(p)$ we use the fact that Theorem \ref{S-estimate} holds with some $C_{p,X,d}>0$ (see \cite{GXHTM1}).

\medskip

We now prove Proposition \ref{S-lem}, 
which is similar to the proof of \cite[Proposition~A.1]{J.M.E.B} while the details are more delicate here.

We will take notations from \cite{J.M.E.B}.
First we consider
$$K^\alpha(x)=\left\{
                \begin{array}{ll}
                  \frac{1}{\Gamma(\alpha)}(1-|x|^2)^{\alpha-1}&,\qquad \hbox{for $|x|<1$,} \\
                 \qquad 0&, \qquad\hbox{for $|x|\geq1,$}
                \end{array}
              \right.$$
and denote its Fourier transform by
$$ m^{\alpha}(\xi):=(K^\alpha)^\wedge(\xi)=\pi^{-\alpha+1}|\xi|^{-d/2-\alpha+1}J_{d/2+\alpha-1}(2\pi|\xi|).$$
Here $J_v$ is the Bessel function of order $v$ and we have the estimate
\begin{equation}\label{estimate for kernel K}
  |m^{\alpha}(\xi)|+|\nabla m^\alpha(\xi)|\leq C_{d,Re(\alpha)}\min(1,|\xi|^{-d/2+1/2-Re(\alpha)}).
\end{equation}
$K^{\alpha}$ is analytic in $\alpha$
and our particular interest is the case $\alpha=0$.%

For $t>0$ we define $K_t^\alpha(x)=t^{-d}K^\alpha(x/t)$, which implies that $(K_t^\alpha)^\wedge(\xi)=m^{\alpha}(t\xi)$.
Let $S_t^\alpha f(x,\omega)=f(\cdot,\omega)\ast K^{\alpha}_t(x)$ and we see $S_t^0=S_t$. One verifies from (\ref{estimate for kernel K}) that
$$\int_0^\8|t\frac{d}{dt}m^{\alpha}(t\xi)|^2\frac{dt}{t}\leq C_{d,Re(\alpha)},$$
for $Re(\alpha)>\frac{3-d}{2}$. Using the Plancherel theorem we have
\begin{equation}\label{estimate for square function in L2}
  \|{\big(\int^{\8}_0|t\frac{d}{dt}S^\alpha_tf|^2\frac{dt}{t}\big)^{\frac{1}{2}}}\|_{L^2(\real^d;L^2(\mu))}\leq C_{d,Re(\alpha)}\norm{f}_{L^2(\real^d;L^2(\mu))}.
\end{equation}
We claim that, if $Re(\alpha)=3$, $X$ is a UMD lattice, and  $1<p<\8$, then
\begin{equation}\label{estimate for square function in Lp}
  \|{\big(\int^{\8}_0|t\frac{d}{dt}S^\alpha_tf|^2\frac{dt}{t}\big)^{\frac{1}{2}}}\|_{L^p(\real^d;X)}\leq \frac{C_{d,p,X}}{\Gamma(\alpha)}\norm{f}_{L^p(\real^d;X)}.
\end{equation}%

We remark again that all these notations and estimates in the scalar form are already contained in \cite{J.M.E.B}.

\medskip

It is well-known that Banach lattice-valued inequalities are closely related to the weighted norm inequalities in harmonic analysis. Rubio de Francia \cite{J.L.F} says that the UMD lattice-valued inequality can be deduced from the weighted norm inequality. 

Let $A_p$ denote the Muckenhoupt class of weights on $\real^d$ and $L^p(w)$ denote the weighted $L^p$ spaces (see e.g. \cite{J.D} for the detailed definition of the weight classes and the weighted Lebesgue space). 

The following lemma is cited from \cite{GXHTM1}.
\begin{lem}\label{lattices-weighted lem}
Let $X$ be a UMD lattice, and let $S$ be a sublinear operator which is bounded on $L^p(\real^d,w)$ for all $w\in A_p$ with $1<p<1+\varepsilon$ for some $\varepsilon>0$. Then $\tilde{S}f(x,\omega)=S(f(\cdot,\omega))(x)$ is bounded on $L^p(\real^d;X)$ for all $1<p<\8$.
\end{lem}

To prove (\ref{estimate for square function in Lp}), by Lemma \ref{lattices-weighted lem}, it is enough to prove the following weighted norm inequality
\begin{equation}\label{weighted norm inequality for vector-valued CZ opator}
  \|{{\big(\int^{\8}_0|t\frac{d}{dt}S^\alpha_tf|^2\frac{dt}{t}\big)^{\frac{1}{2}}}}\|_{L^p(w)}\leq
  \frac{C_{d,p}}{\Gamma(\alpha)}\norm{f}_{L^p(w)},~\forall~w\in A_p.
\end{equation}
To prove (\ref{weighted norm inequality for vector-valued CZ opator}), we redefine the left hand side of (\ref{weighted norm inequality for vector-valued CZ opator}).
Set
$$Tf=f\ast \mathcal{K},\qquad where\qquad \bigg(\mathcal{K}(x)=(t\frac{d}{dt}K^\alpha_t(x)):t>0\bigg)$$
as an operator $T$ whose kernel is valued in the Hilbert space $H=L^2\big((0,\8),\frac{dt}{t}\big)$. 
Then
$$ \|{{\big(\int^{\8}_0|t\frac{d}{dt}S^\alpha_tf|^2\frac{dt}{t}\big)^{\frac{1}{2}}}}\|_{L^p(w)}=\norm{\|Tf\|_{H}}_{L^p(w)}$$

We check that $T$ is a vector-valued Calder\'{o}n-Zygmund operator. By (\ref{estimate for kernel K}), we can see that $T$ is bounded in $L^2(\real^d)$. According to the definition of $K^{\alpha}_t$ and $Re(\alpha)=3$, by a straightforward computation we get
$$\norm{\mathcal{K}(x)}_{\mathbb{C}\rightarrow H}\leq \frac{C_d}{|\Gamma(\alpha)|}|x|^{-d},\qquad \norm{\nabla_x\mathcal{K}(x)}_{\mathbb{C}\rightarrow H}\leq\frac{C_d}{|\Gamma(\alpha)|}|x|^{-d-1}.$$
By the weighted norm inequalities for square functions (cf. \cite{HH}), (\ref{estimate for square function in Lp}) is proved.

\medskip

Next fix a $p$,  $1<p<\8$. We will prove that there exists a $d_0$ depending on $p$ and $X$ such that
\begin{equation}\label{inequality of square function in fixed dem}
 \|{\big(\int^{\8}_0|t\frac{d}{dt}S_tf|^2\frac{dt}{t}\big)^{\frac{1}{2}}}\|_{L^p(\real^{d_0};X)}\leq C_{p,X}\norm{f}_{L^p(\real^{d_0};X)}.
\end{equation}

For UMD lattice $X$ in $(\Omega, \mu)$, there exists $\theta_0$, $0<\theta_0<1$, and another UMD lattice $Y$, such that $X=[L^2(\mu),Y]_{\theta_0}$ (see \cite[page 251]{J.L.F}). For vector-valued interpolation, the reader can see \cite{BL} for more details.%


Similar to \cite{J.M.E.B}, we will apply vector-valued complex interpolation to the analytic family of operators
$$ \mathcal{S}^{\alpha}f=(e^{\alpha^2}t\frac{d}{dt}S_t^{\alpha} f:t>0).$$
Recall that we have obtained from  (\ref{estimate for square function in L2}) and (\ref{estimate for square function in Lp}) that
\begin{equation*}
  \begin{split}
  & \norm{\mathcal{S}^\alpha}_{L^2(\real^d;L^2(H))\rightarrow L^2(\real^d;L^2(\mu))}\leq C_{d,\varepsilon}\quad \textit{for}\quad Re(\alpha)=(3-d)/2+\varepsilon,\\
  & \norm{\mathcal{S}^\alpha}_{L^{2}(\real^d;Y(H))\rightarrow L^{2}(\real^d;Y)}\leq C_{d,X}\qquad \textit{for}\quad Re(\alpha)=3.
   \end{split}
\end{equation*}
Here and in the sequel, the inequality $\|\mathcal{S}^\alpha\|_{A\to B}\le C$ means $\|\mathcal{S}^\alpha f\|_{A}\le C\|f\|_{B}$.  For the $\theta_0$, and a small number $\varepsilon>0$ we have
\begin{equation}\label{L^q inequality}
 \|{\mathcal{S}^{\alpha}}\|_{L^{2}(\real^d;X(H))\rightarrow L^{2}(\real^d;X)}\leq C_{d,\varepsilon,X},
\end{equation}
where
\begin{align*}
 &Re(\alpha)=(1-\theta_0)((3-d)/2+\varepsilon)+3\theta_0.
\end{align*}

We are ready to prove (\ref{inequality of square function in fixed dem}). Assume first that $1<p<2$. Take $1<q<2$, where $q$ is close to 1 and write ${1}/{p}={(1-\theta)}/{2}+{\theta}/{q}$, where $0<\theta<1$. The quantities $q$, $\theta$ and $\varepsilon$ are to be determined. We use again vector-valued complex interpolation for the analytic family of operators
$$ \mathcal{S}^{\alpha}f=(e^{\alpha^2}t\frac{t}{dt}S_t^{\alpha} f:t>0).$$
By (\ref{estimate for square function in Lp}) and (\ref{L^q inequality}), we have
\begin{equation}\label{interpolation equalities}
  \begin{split}
  & \norm{\mathcal{S}^\alpha}_{L^2(\real^d;X(H))\rightarrow L^2(\real^d;X)}\leq C_{d,\varepsilon,X},~  Re(\alpha)=(1-\theta_0)((3-d)/2+\varepsilon)+3\theta_0,\\
   & \|{\mathcal{S}^{\alpha}}\|_{L^{q}(\real^d;X(H))\rightarrow L^{q}(\real^d;X)}\leq C_{d,X},\quad Re(\alpha)=3.\\
   \end{split}
\end{equation}

Let $\varepsilon>0$
be the number such that
\begin{align*}
 &(1-\theta)\{(1-\theta_0)((3-d)/2+\varepsilon)+3\theta_0\}+3\theta=0.
\end{align*}
Then
\begin{equation}\label{interpolation equality}
 p=((1-\theta)/2+\theta/q)^{-1},~p>\frac{(1-\theta_0)(d+3)}{(d-3-(d+3)\theta_0)/q+3}.
\end{equation}
This combined with a standard complex interpolation implies that (\ref{inequality of square function in fixed dem}) holds when
$$d_0=d_0(p)>\frac{3}{(1-\theta_0)(p-1)}+\frac{3\theta_0}{1-\theta_0}.$$
A similar argument for the case $p\ge 2$ works.
In summary
we are able to prove (\ref{inequality of square function in fixed dem})  for $d_0=d_0(p)=[\max\{\frac{3}{(1-\theta_0)(p-1)}+\frac{3\theta_0}{1-\theta_0},\frac{3}{(1-\theta_0)(p^{\prime}-1)}+\frac{3\theta_0}{1-\theta_0}\}]$, where $[a]$ denotes the largest integer $\le a$.
We refer the readers to \cite[P97]{J.M.E.B} for more details about this part.
Moreover if we follow the argument on \cite[P97]{J.M.E.B} by replacing the scale-valued functions by vector-valued functions,
we will finish the proof of Proposition~\ref{S-lem}.

\medskip
\bigskip
\section{the estimate of the long variation}\label{ST3}

In this section we estimate the long variation $\mathcal{LV}_{q}(\mathcal{A})f$.
Recall that
$$\mathcal{LV}_{q}(\mathcal{A})f= V_q(A_tf(x,\omega):t\in\{2^k:k\in \mathbb{Z}\}).$$
We prove the following theorem.
\begin{thm}\label{L-estimate}
Let $X$ be a UMD lattice. Let $1<p<\8$, then there exists a constant $C_{p,X}>0$ independent of dimension $d$ such that
\begin{equation}\label{L-inequalities}
\norm{\mathcal{LV}_q(\mathcal{A})f}_{L^p(\real^d;X)}\leq C_{p,q,X}\norm{f}_{L^p(\real^d;X)},\quad\forall f\in L^p(\real^d;X).
\end{equation}
\end{thm}

Before proving the theorem, we need some notations and lemmas. 

Let $P_t$ be the Poisson semigroup, and denote $P_tf(x,\omega)=P_t\ast f(\cdot,\omega)(x)$. Then for every $\xi\in\real^d$, we have $(P_tf)^\wedge(\xi,\omega)=p_t(\xi)\hat{f}(\cdot,\omega)(\xi)$, where $p_t(\xi)=e^{-2\pi t|\xi|}$.%

For $f\in L^p(\real^d;X)$, let us introduce the maximal function
$$P^{\ast}f(x,\omega)=\sup_{t>0}|P_tf|(x,\omega),$$
and the square function
$$g(f)(x,\omega)=\bigg(\int_0^\8t|\frac{d}{dt}P_tf(x,\omega)|^2dt\bigg)^{1/2},$$
associated with the Poisson semigroup.%

According to \cite[Theorem 1, p. 46]{E.M.S}, we know that Poisson semigroup $P_t$ is a contractively regular operator, with $1<p<\8$, and $P_t$ is strong continuous and analytic.

To prove Theorem \ref{L-estimate}, we need some lemmas. 

In \cite{GXHTM1}, Hong and Ma establish UMD lattice-valued square functions and variational inequalities for analytic semigroups. Here we state these results for the Possion semigroup.
\begin{lem}\label{g function estimate}
Let $X$ be any UMD lattice. For every $1<p<\8$ there exists a constant $C_{p,X}>0$ independent of the dimension such that for every $f\in L^p(\real^d;X)$ we have
\begin{equation}\label{square function associate possion function}
  \|{\big(\int_0^\8t|\frac{d}{dt}P_tf(x,\omega)|^2dt\big)^{1/2}}\|_{L^p(\real^d;X)}
\leq C_{p,X}\norm{f(x,\omega)}_{L^p(\real^d;X)}
\end{equation}
\end{lem}
\begin{lem}\label{Poisson semigroup q-variation}
Let $X$ be any UMD lattice. For every $1<p<\8$, there exists a constant $C_{p,q,X}$ independent of the dimension such that for every $f\in L^p(\real^d;X)$ we have
$$\norm{V_q(P_tf:t>0)}_{L^p(\real^d;X)}\leq C_{p,q,X}\norm{f}_{L^p(\real^d;X)}.$$
\end{lem}
In \cite{Xu}, Xu establishes UMD lattice-valued maximal functions for analytic semigroups. Here we state the result for the Possion semigroup.
\begin{lem}\label{maximal inequality}
Let $X$ be any UMD lattice. For every $1<p<\8$, there exists a constant $C_{p,X}$ independent of the dimension such that for every $f\in L^p(\real^d;X)$ we have
$$\norm{P^*f}_{L^p(\real^d;X)}\leq C_{p,X}\norm{f}_{L^p(\real^d;X)}.$$
\end{lem}
For UMD lattice-valued Hardy-Littlewood maximal operator
$$Mf(x,\omega)=\sup_{t>0}A_t|f|(x,\omega),$$
Deleaval and Kriegler \cite{D.K} have proved the following lemma.
\begin{lem}\label{M.L inequality}
For every $1<p<\8$, there exists a constant $C_{p,X}$ independent of the dimension $d$ such that for every $f\in L^p(\real^d;X)$ we have
$$\norm{Mf}_{L^p(\real^d;X)}\leq C_{p,X}\norm{f}_{L^p(\real^d;X)}.$$
\end{lem}
\medskip

Let $B$ be a Euclidean ball, we assume that $|B|=1$. By \cite{Bour1}, there is an isotropic constant $L=L(B)>0$ such that for every unit vector $\xi\in\mathbb{S}^{d-1}$ we have
\begin{equation}\label{isotropic constant L}
  L=\big(\int_{B}(\xi\cdot x)^2dx\big)^{1/2}.
\end{equation}
Denote $m(\xi)=\int_Be^{-2\pi ix\cdot\xi}dx$, then we have the following proposition (cf.\cite[eq. (10),(11)]{Bour1}).
\begin{prop}\label{constant m}
Let $B$ be a Euclidean ball on $\real^d$. Let $|B|=1$, then there exists a constant $C>0$ independent of the dimension $d$ such that for every $\xi\in\real^d$ we have
$$|m(\xi)|\le C(L|\xi|)^{-1},\qquad |m(\xi)-1|\le CL|\xi|.$$
\end{prop}

Associating with the Poisson semigroup, for every $n\in\mathbb{Z}$, we defined the projections $W_n=P_{L2^n}-P_{L2^{n-1}}$, where $L=L(B)$ is isotropic constant which is defined in \eqref{isotropic constant L}. By this definition, we can see that, for every $f\in\mathcal{S}(\real^d)\otimes X$,
\begin{equation}\label{function decomposition}
  f=\sum_{n\in\mathbb{Z}}W_nf,\qquad\textit{for}~a.e.~(x,\omega)\in \real^d\times\Omega.
\end{equation}

For each $n\in\mathbb{Z}$, we have
$$W_nf(x,\omega)=\int_{L2^{n-1}}^{L2^n}\frac{d}{dt}P_tf(x,\omega)dt.$$
Thus, by the Cauchy-Schwarz inequality we get
\begin{equation*}
  \begin{split}
 |W_nf(x,\omega)|^2 &\leq\bigg(\int_{L2^{n-1}}^{L2^n}|\frac{d}{dt}P_tf(x,\omega)|dt\bigg)^2\\
&\leq L2^{n-1}\int_{L2^{n-1}}^{L2^n}|\frac{d}{dt}P_tf(x,\omega)|^2dt\\
&\leq\int_{L2^{n-1}}^{L2^n}t|\frac{d}{dt}P_tf(x,\omega)|^2dt.
   \end{split}
\end{equation*}

By Lemma \ref{g function estimate}, we have the following lemma.
\begin{lem}\label{Littlewood-paley inequalities}
For every $1<p<\8$ there exists a constant $C_{p,X}>0$ independent of the dimension such that for every $f\in L^p(\real^d;X)$ we have
$$\|{\big(\sum_{n\in\mathbb{Z}}|W_nf|^2\big)^{1/2}}\|_{L^p(\real^d;X)}\leq C_{p,X}\norm{f}_{L^p(\real^d;X)}.$$
\end{lem}

\medskip


We are now ready to prove Theorem \ref{L-estimate} using previous lemmas and ideas from \cite[Section 3]{J.M.E.B}.
We contain this proof below for the sake of completeness, however we delete some standard details
which could be found in \cite[Section 3]{J.M.E.B}.

\begin{proof}[Proof of Theorem \ref{L-estimate}]

We have the following control
\begin{equation}\label{L-decomposition inequality}
  \begin{split}
\norm{V_q(A_{2^n})f:n\in\mathbb{Z}}_{L^p(\real^d;X)}\leq \norm{V_q(P_{L2^n}f:n\in\mathbb{Z})}_{L^p(\real^d;X)} \\
+\|{\big(\sum_{n\in\mathbb{Z}}|A_{2^n}f-P_{L2^n}f|^2\big)^{1/2}}\|_{L^p(\real^d;X)}.
  \end{split}
\end{equation}
It suffices to estimate the second term on the right hand side while the first term is
bounded on $L^p(\real^d;X)$ by Lemma \ref{Poisson semigroup q-variation}.
Using \eqref{function decomposition}, we need to show that
\begin{equation}\label{L-inequalities 1}
  \begin{split}
\|{\big(\sum_{n\in\mathbb{Z}}|(A_{2^n}W_{j+n}f(x,\omega)-P_{L2^n}W_{j+n}f(x,\omega))|^2\big)^{1/2}}\|_{L^p(\real^d;X)}\\
\leq C_{p,X} \sum_{j\in \mathbb{Z}}2^{-\delta_p|j|}\|f(x,\omega)\|_{L^p(\real^d;X)}
  \end{split}
\end{equation}
since this implies that
\begin{equation*}
  \begin{split}
&\|{\big(\sum_{n\in\mathbb{Z}}|A_{2^n}f(x,\omega)-P_{L2^n}f(x,\omega)|^2\big)^{1/2}}\|_{L^p(\real^d;X)}\\
\leq &C_{p,X}\|f(x,\omega)\|_{L^p(\real^d;X)}.
  \end{split}
\end{equation*}

From now on we focus on (\ref{L-inequalities 1}),
which follows from interpolation and the following two inequalities:
\begin{equation}\label{L-inequalities 2}
\begin{split}
 \|{\big(\sum_{n\in\mathbb{Z}}|A_{2^n}W_{j+n}f(x,\omega)|^2\big)^{1/2}}\|_{L^p(\real^d;Z)}+&
\|{\big(\sum_{n\in\mathbb{Z}}|P_{L2^n}W_{j+n}f(x,\omega))|^2\big)^{1/2}}\|_{L^p(\real^d;Z)}\\
&\leq C_{p,Z} \norm{f(x,\omega)}_{L^p(\real^d;Z)},
\end{split}
\end{equation}
for any UMD lattice $Z$,
and
\begin{equation}\label{L-inequalities 3}
\begin{split}
  \|{\big(\sum_{n\in\mathbb{Z}}|(A_{2^n}W_{j+n}f(x,\omega)-P_{L2^n}W_{j+n}f(x,\omega))|^2\big)^{1/2}}\|_{L^2(\real^d;L^2(\mu))}\\
\leq C 2^{-|j|/2}\norm{f(x,\omega)}_{L^2(\real^d;L^2(\mu))}.
\end{split}
\end{equation}

Now under the assumption that (\ref{L-inequalities 2}) and (\ref{L-inequalities 3}) are true, we prove that the inequality (\ref{L-inequalities 1}) is true.

Since there exists a $\theta_0\in(0,1)$, such that $X=[L^2(\mu),Y]_{\theta_0}$, where $Y$ is another lattice. Let $Z=Y$ in (\ref{L-inequalities 2}),  interpolating with (\ref{L-inequalities 3}),
we have 
\begin{equation}\label{L-inequalities p0}
\begin{split}
  \|\big(\sum_{n\in\mathbb{Z}}|\sum_{j\in\mathbb{Z}}(A_{2^n}W_{j+n}f-P_{L2^n}W_{j+n}f)|^2\big)^{1/2}\|_{L^{2}(\real^d;X)}\\
  \leq C_{X} 2^{-(1-\theta_0)|j|/2}\norm{f(x,\omega)}_{L^{2}(\real^d;X)}.
  \end{split}
\end{equation}

Then let $Z=X$ in (\ref{L-inequalities 2}), interpolating with (\ref{L-inequalities p0}),
we have
\begin{equation*}
\begin{split}
  \|\big(\sum_{n\in\mathbb{Z}}|\sum_{j\in\mathbb{Z}}(A_{2^n}W_{j+n}f-P_{L2^n}W_{j+n}f)|^2\big)^{1/2}&\|_{L^{p}(\real^d;X)}\\
  \leq C_{p,X} 2^{-(1-\theta_0)(1-\theta)|j|/2}\norm{f(x,\omega)}_{L^p(\real^d;X)},
\end{split}
\end{equation*}
where $\frac{1}{p}=\frac{1-\theta}{2}+\frac{\theta}{q}$.
This proved that (\ref{L-inequalities 1}) is true for any $1<p<\8$. Then the remaining task is to show inequalities (\ref{L-inequalities 2}) and (\ref{L-inequalities 3}).

By Lemma~\ref{Littlewood-paley inequalities},  estimate  (\ref{L-inequalities 2}) is obviously a consequence of the following two estimates,
\begin{equation}\label{L-inequalities 4}
\|{\big(\sum_{n\in\mathbb{Z}}|A_{2^n}g_n(x,\omega)|^2\big)^{1/2}}\|_{L^p(\real^d;Z)}
\leq C_{p,Z}\|{\big(\sum_{n\in\mathbb{Z}}|g_n(x,\omega)|^2\big)^{1/2}}\|_{L^p(\real^d;Z)}
\end{equation}
and
\begin{equation}\label{L-inequalities 5}
  \|{\big(\sum_{n\in\mathbb{Z}}|P_{L2^n}g_{n}(x,\omega)|^2\big)^{1/2}}\|_{L^p(\real^d;Z)}
\leq C_p \|{\big(\sum_{n\in\mathbb{Z}}|g_{n}(x,\omega)|^2\big)^{1/2}}\|_{L^p(\real^d;Z)}
\end{equation}
for all $p\in(1,\8)$.
Estimates \eqref{L-inequalities 4} and \eqref{L-inequalities 5} are proved using complex interpolation,
the duality identity $C(p,s,Z)=C(p^{\prime},s^{\prime},Z^*)$, together with Lemma \ref{maximal inequality} and Lemma \ref{M.L inequality} .

The estimate \eqref{L-inequalities 3} follows from the argument in \cite[Section 3]{J.M.E.B} if we replace
scale-valued functions by vector-valued functions in the argument.
We remark that Proposition~\ref{constant m} is used in this process.

This finishes the proof of Theorem \ref{L-inequalities}.

\medskip

\end{proof}

\noindent \textbf{Acknowledgment.} Guixiang Hong is partially supported by the NSF of China (grant no. 11601396, 11501169, 11431011), and 1000 Young Talent Researcher Program of China (no. 429900018- 101150(2016)). 
Danqing He is supported by NNSF of China (No. 11701583),  the Guangdong Natural Science Foundation
(No. 2017A030310054), and the
Fundamental Research Funds for the Central Universities (No. 17lgpy11).

\medskip
\bigskip

\end{document}